\documentclass{article}

\usepackage{microtype}
\usepackage{graphicx}
\usepackage{subfigure}
\usepackage{booktabs} 
\usepackage{amsmath,amsthm,amsfonts,amssymb,verbatim}

\usepackage{hyperref}
\hypersetup{draft}

\usepackage[accepted]{icml2020}

\icmltitlerunning{Deep Neural Networks with Binary Activation Functions}
\newtheorem{theorem}{Theorem}
\newtheorem{remark}[theorem]{Remark}
\newtheorem{lemma}[theorem]{Lemma}
\usepackage{bbm}
\def\R{\mathbb{R}}

\begin{document}
%
%

\twocolumn[
\icmltitle{An Integer Programming Approach to Deep Neural Networks\\ with Binary Activation Functions}



\icmlsetsymbol{equal}{*}

\begin{icmlauthorlist}
	\icmlauthor{Jannis Kurtz}{ed}
	\icmlauthor{Bubacarr Bah}{to,goo}
\end{icmlauthorlist}

\icmlaffiliation{ed}{Chair for Mathematics of Information Processing, RWTH Aachen University, Aachen, Germany}
\icmlaffiliation{to}{African Institute for Mathematical Sciences, Cape Town, South Africa}
\icmlaffiliation{goo}{Division of Applied Mathematics, Stellenbosch University, Stellenbosch, South Africa}

\icmlcorrespondingauthor{Jannis Kurtz}{kurtz@mathc.rwth-aachen.de}

\icmlkeywords{Discrete Neural Networks, Binary Activation Functions, Adversarial Attacks}

\vskip 0.3in
]



\printAffiliationsAndNotice{\icmlEqualContribution} 

\begin{abstract}
We study deep neural networks with binary activation functions (BDNN), i.e. the activation function only has two states. We show that the BDNN can be reformulated as a mixed-integer linear program which can be solved to global optimality by classical integer programming solvers. Additionally, a heuristic solution algorithm is presented and we study the model under data uncertainty, applying a two-stage robust optimization approach. We implemented our methods on random and real datasets and show that the heuristic version of the BDNN outperforms classical deep neural networks on the Breast Cancer Wisconsin dataset while performing worse on random data.
\end{abstract}

\section{Introduction}
Deep learning (DL) methods have of-late reinvigorated interest in artificial intelligence and data science, and they have had many successful applications in computer vision, natural language processing, and data analytics \cite{lecun2015deep}. The training of deep neural networks relies mostly on (stochastic) gradient descent, hence the use of differentiable activation functions like ReLU, sigmoid or the hyperbolic tangent is the state of the art \cite{rumelhart1986learning,goodfellow2016deep}. On the contrary discrete activations, which may be more analogous to biological activations, present a training challenge due to non-differentiability and even discontinuity. If additionally the weights are considered to be binary, the use of discrete activation networks can reduce the computation and storage complexities, provides for better interpretation of solutions, and has the potential to be more robust to adversarial perturbations than the continuous activation networks \cite{qin2020binary}. Furthermore low-powered computations may benefit from discrete activations as a form of coarse quantizations \cite{plagianakos2001training,bengio2013estimating,courbariaux2015binaryconnect,rastegari2016xnor}. Nevertheless, gradient descent-based training behaves like a black box, raising a lot of questions regarding the explainability and interpretability of internal representations \cite{hampson1990representing,plagianakos2001training,bengio2013estimating}.  

On the other hand, integer programming (IP) is known as a powerful tool to model a huge class of real-world optimization problems \cite{wolsey1998integer}. Recently it was successfully applied to machine learning problems involving sparsity constraints and to evaluate trained neural networks \cite{bertsimas2017sparse,bertsimas2019sparse,fischetti2018deep}.

\paragraph{Contributions:}
In Section \ref{sec:DNN}, we reformulate the binary neural network as an IP formulation, referred to in the sequel as a binary deep neural network (BDNN), which can be solved to global optimality by classical IP solvers. All our results can also be applied to BDNNs with binary weights. Solving the latter class of networks to optimality was still an open problem. We note that it is straight forward to extend the IP formulation to more general settings and several variations of the model. We present a heuristic algorithm to calculate local optimal solutions of the BDNN in Section \ref{sec:heuristic} and apply a two-stage robust optimization approach to protect the model against adversarial attacks in Section \ref{sec:adversarialAttacks}. In Section \ref{sec:computations} we present computational results for random and real datasets and compare the BDNN to a deep neural network using ReLU activations (DNN). Despite scalability issues and a slightly worse accuracy on random datasets, the results indicate that the heuristic version outperforms the DNN on the \textit{Breast Cancer Wisconsin} dataset. 

\section{Related Literature}
The interest in BDNN goes back to \cite{mcculloch1943logical} where BDNNs were used to simulate Boolean functions. However, until the beginning of this century concerted efforts were made to train these networks either by specific schemes \cite{gray1992training,kohut2004boolean} or via back propagation by modifications of the gradient descent method \cite{widrow1988neural,toms1990training,barlett1992using,goodman1994learning,corwin1994iterative,plagianakos2001training}. More recent work regarding the back propagation, mostly motivated by the low complexity of computation and storage, build-on the pioneering works of \cite{bengio2013estimating,courbariaux2015binaryconnect,hubara2016binarized,rastegari2016xnor,kim2016bitwise}; see \cite{qin2020binary} for a detailed survey. Regarding the generalization error of BDNN, it was already proved that the VC-Dimension of deep neural networks with binary activation functions is $\rho\log (\rho)$, where $\rho$ is the number of weights of the BDNN; see \cite{baum1989size,maass1994perspectives,sakurai1993tighter}.

One the other hand IP methods were successfully applied to sparse classification problems \cite{bertsimas2017sparse}, sparse regression \cite{bertsimas2019sparse} and to evaluate trained neural networks \cite{fischetti2018deep}. In \cite{icarte2019training} BDNNs with weights restricted to $\{-1,1\}$ are trained by a hybrid method based on constraint programming and mixed-integer programming. In \cite{khalil2018combinatorial}  the authors calculate optimal adversarial examples for BDNNs using a MIP formulation and integer propagation. Furthermore, robust optimization approaches were used to protect against adversarial attacks for other machine learning methods \cite{xu2009robustness,xu2009robust,bertsimas2019robust}.

\section{Discrete Neural Networks}\label{sec:DNN}
In this work we study \textit{binary deep neural networks} (BDNN), i.e. classical deep neural networks with binary activation functions. As in the classical framework, for a given input vector $\mathop{x\in \R^n}$ we study classification functions $f$ of the form $\mathop{f(x)=\sigma^K\left( W^K\sigma^{K-1}\left( W^{K-1}\ldots ,\sigma^1\left(W^1x\right)\ldots \right)\right)}$ for weight matrices $W^k\in \R^{d_{k}\times d_{k-1}}$ and activation functions $\sigma^k$, which are applied component-wise. The dimension $d_{k}$ is called the \textit{width} of the \textit{$k$-th layer} . In contrast to the recent developments of the field we consider the activation functions to be binary, more precisely each function is of the form
\begin{equation}\label{eq:defActivationFunction}
\sigma^k(\alpha ) = \begin{cases}
0 & \text{ if } \alpha < \lambda_k \\
1 & \text{ otherwise}
\end{cases} 
\end{equation}

for $\alpha\in \R$ where the parameters $\lambda_k\in \R$ can be learned by our model simultaneously with the weight matrices which is normally not the case in the classical neural network approaches. Note that it is also possible to fix the values $\lambda_k$ in advance.

In the following we use the notation $[p]:=\left\{ 1,\ldots ,p\right\}$ for $p\in \mathbb N$.

Given a set of labeled training samples \[X\times Y = \left\{ (x^i,y^i) \ | \ i\in [m]\right\}\subset \R^n\times \left\{ 0,1\right\}\] we consider loss functions
\begin{equation}\label{eq:defLossFunction}
\ell: \left\{ 0,1 \right\} \times \R^{d_K} \to \R
\end{equation}
and the task is to find the optimal weight matrices which minimize the empirical loss over the training samples, i.e. we want to solve the problem
\begin{equation}\label{eq:DNNDefinition}
\begin{aligned}
\min \ & \sum_{i=1}^{m} \ell\left( y^i , z^i\right) \\
s.t. \quad & z^i = \sigma^K\left( W^K\sigma^{K-1}\left(\ldots \sigma^1\left(W^1x^i\right)\ldots \right)\right) \ \ \forall i\in [m] \\
& W^k\in\R^{d_{k}\times d_{k-1}} \ \ \forall k\in[K] \\
& \lambda_k\in \R \ \ \forall k\in [K]
\end{aligned}
\end{equation}
for given dimensions $d_0,\ldots d_{K}$ where $d_0=n$. We set $d_{K}=:L$. All results in this work even hold for \textit{regression problems}, more precisely for loss functions
\[
\ell_r: \R^L \times \R^L \to \R
\]
where we minimize the empirical loss $\sum_{i=1}^{m} \ell_r\left( x^i , z^i\right)$ instead. We use labels $\left\{ 0,1\right\}$ instead of the classical labels $\left\{ -1, 1\right\}$ due to ease of notation in our model. Nevertheless the same approach can be adapted to labels $\tilde y^i \in \left\{ -1, 1\right\}$ by adding the additional constraints
\[
\tilde y^i = -1 + 2y^i.
\]

Several variants of the model can be considered:
\begin{itemize}
	\item The output of the neural network is $0$ or $1$, defining the class the input is assigned to. As a loss function, we count the number of misclassified training samples. This variant can be modeled by choosing the last layer to have width $d_{K} = 1$ together with the loss function $\ell (y , z ) = |y - z|$.
	
	\item The output of the neural network is a continuous vector $z\in \R^L$. This case can be modeled by choosing the activation function of the last layer $\sigma^K$ as the identity and $d_K=L$. Any classical loss functions maybe consider, e.g. the squared loss $\ell (x, z ) = \| y - z\|_2^2$ .
	Note that a special case of this model is the \textit{autoencoder} where in contrast to the continuous framework the input is encoded into a vector with $0$-$1$ entries.

	\item All results can easily be generalized to multiclass classification; see Section \ref{sec:DNN}.
\end{itemize} 

In the following lemma we show how to reformulate Problem \eqref{eq:DNNDefinition} as an integer program.
\begin{lemma}\label{lem:MINLP_formulation}
Assume the euclidean norm of each data point in $X$ is bounded by $r$, then Problem \eqref{eq:DNNDefinition} is equivalent to the mixed-integer non-linear program
\begin{align}
& \min \ ~\sum_{i=1}^{m} \ell\left( y^i , u^{i,K}\right) \quad s.t. \ \  \label{eq:MINLP_objective}\\
& W^1 x^i < M_1 u^{i,1} + \lambda_1  \label{eq:MINLP_firstlayer}\\
& W^1 x^i \ge M_1 (u^{i,1}-1) + \lambda_1 \label{eq:MINLP_firstlayer2}\\
& W^k u^{i,k-1} < M_k u^{i,k} + \lambda_k  \ \ \forall k\in [K]\setminus \{1\}\label{eq:MINLP_layers}\\
& W^k u^{i,k-1} \ge M_k(u^{i,k}-1) + \lambda_k  \ \ \forall k\in [K]\setminus \{1\}\label{eq:MINLP_layers2}\\
& W^k \in[-1,1]^{d_k\times d_{k-1}}, \ \ \lambda_k\in[-1,1] \ \ \forall k\in [K]\\ 
& u^{i,k}\in \left\{ 0,1\right\}^{d_k} \ \ \forall k\in [K], i\in [m],\label{eq:MINLP_variables}
\end{align}
where $M_1:=(nr+1)$ and $M_k:=(d_{k-1}+1)$.
\end{lemma}

\begin{proof}
First we show that, due to the binary activation functions, we may assume $W^k\in[-1,1]^{d_k\times d_{k-1}}$ and $\lambda_k\in [-1,1]$ for all $k\in [K]$ in our model. To prove this assume we have any given solution $W^1,\ldots , W^K$ and corresponding $\lambda_1,\ldots ,\lambda_K$ of problem \eqref{eq:DNNDefinition} with arbitrary values in $\R$. Consider any fixed layer $k\in [K]$. The $k$-th layer receives a vector $h^{k-1}\in\{ 0, 1\}^{d_{k-1}}$ from the previous layer, which is applied to $W^k$ and afterwards the activation function $\sigma^k$ is applied component-wise, i.e. the output of the $k$-th layer is a vector
\[
h^k_j = \begin{cases} 0 & \text{ if } (w_j^k)^\top h^{k-1}<\lambda_k \\ 1 & \text{ otherwise}\end{cases}
\]
where $w_j^k$ is the $j$-th row of the matrix $W^k$. Set 
\[\beta:=\max\{|\lambda_k|,\max_{\substack{j=1,\ldots ,d_k\\ l=1,\ldots ,d_{k-1}}}|w_{jl}^k|\}\]
and define $\tilde W^k:=\beta W^k$ and $\tilde \lambda_k:=\beta\lambda_k$. Then replacing $W^k$ by $\tilde W^k$ and $\lambda_k$ by $\tilde \lambda_k$ in the $k$-th layer yields the same output vector $h^k$, since the inequality $ (w_j^k)^\top h^{k-1}<\lambda_k$ holds if and only if the inequality $ (\tilde w_j^k)^\top h^{k-1}<\tilde \lambda_k$ holds. Furthermore all entries of $\tilde W^k$ and $\tilde \lambda_k$ have values in $[-1,1]$.

Next we show that the constraints \eqref{eq:MINLP_firstlayer}--\eqref{eq:MINLP_layers2} correctly model the equation 
\begin{align}
z^i	:= & ~u^{i,K} \nonumber\\
	= & ~\sigma^K\left( W^K\sigma^{K-1}\left( W^{K-1}\ldots ,\sigma^1\left(W^1x^i\right)\ldots \right)\right) \nonumber
\end{align}
of Problem \eqref{eq:DNNDefinition}. The main idea is that the $u^{i,k}$-variables model the output of the activation functions of data point $i$ in layer $k$, i.e. they have value $0$ if the activation value is $0$ or value $1$ otherwise. More precisely for any solution $W^1, \ldots, W^k$ and $\lambda_1,\ldots ,\lambda^k$ of the Problem in Lemma \ref{lem:MINLP_formulation} the variable $u_j^{i,1}$ is equal to $1$ if and only if $(w_j^1)^\top x^i\ge \lambda_1$ since otherwise Constraint \eqref{eq:MINLP_firstlayer2} would be violated. Note that if $u_j^{i,1}=1$, then Constraint \eqref{eq:MINLP_firstlayer} is always satisfied since all entries of $W^1$ are in $[-1,1]$ and all entries of $x^i$ are in $[-r,r]$ and therefore $|W^1 x^i|\le nr < M_1$. Similarly we can show that $u_j^{i,1}=0$ if and only if $(w_j^1)^\top x^i < \lambda_1$. Hence $u^{i,1}$ is the output of the first layer for data point $x^i$ which is applied to $W^2$ in Constraints \eqref{eq:MINLP_layers} and \eqref{eq:MINLP_layers2}. By the same reasoning applied to Constraints \eqref{eq:MINLP_layers} and \eqref{eq:MINLP_layers2} we can show that $u^{i,k}$ is equal to the output of the $k$-th layer for data point $x^i$ for each $k\in [K]\setminus \{ 1\}$. Note that instead of the value $nr$ we can use $d_{k-1}$ here since the entries of $u^{i,k-1}$ can only have values $0$ or $1$ and the dimension of the rows of $W^k$ is $d_{k-1}$.
\end{proof}

The formulation in Lemma \ref{lem:MINLP_formulation} is a non-linear mixed-integer programming (MINLP) formulation, since it contains products of variables, where each is a product of a continuous variable and an integer variable. By linearizing these products we can prove the following theorem.
\begin{theorem}\label{thm:MILP_formulation}
Under the assumption of Lemma \ref{lem:MINLP_formulation} Problem \eqref{eq:DNNDefinition} is equivalent to the mixed-integer program
\begin{equation*}\label{eq:MIP_formulation}
\begin{aligned}
(\text{MIP}&): ~\min  \sum\limits_{i=1}^{m} \ell\left( y^i , u^{i,K}\right) \quad s.t. \\
& W^1 x^i < M_1 u^{i,1} + \lambda_1 \quad  \forall i\in [m]\\
& W^1 x^i \ge M_1 (u^{i,1}-1) + \lambda_1 \quad  \forall i\in [m]\\
&\sum_{l=1}^{d_{k-1}}s_l^{i,k} < M_k u^{i,k} + \lambda_k \\
& \qquad\forall i\in [m], k\in [K]\setminus \{1\}\\
&\sum_{l=1}^{d_{k-1}}s_l^{i,k} \ge M_k (u^{i,k}-1) + \lambda_k \\
& \qquad \ \forall i\in [m], k\in [K]\setminus \{1\}\\
&s_{lj}^{i,k} \le u_j^{i,k}, \quad s_{lj}^{i,k} \ge -u_j^{i,k} \\
&\qquad \forall i\in [m], k\in [K]\setminus\{ 1\}, l\in [d_{k-1}], j\in [d_k] \\
&s_{lj}^{i,k} \le w_{lj}^k + (1-u_j^{i,k}) \\
& \qquad \forall i\in [m], k\in [K]\setminus\{ 1\}, l\in [d_{k-1}], j\in [d_k]\\
&s_{lj}^{i,k} \ge w_{lj}^k - (1-u_j^{i,k}) \\
& \qquad \forall i\in [m], k\in [K]\setminus\{ 1\}, l\in [d_{k-1}], j\in [d_k]\\
& W^k\in[-1,1]^{d_k\times d_{k-1}} \quad \forall k\in [K]\\
& \lambda_k\in [-1,1] \quad \forall k\in [K]\\
& u^{i,k}\in \left\{ 0,1\right\}^{d_k} \ \ \forall k\in [K], i\in [m]\\
& s_l^{i,k}\in [-1,1]^{d_k} \ \ \forall i\in [m], k\in [K]\setminus\{ 1\}, l\in [d_{k-1}]
\end{aligned}
\end{equation*}
\end{theorem}

\begin{proof}
We replace each product of variables $w_{lj}^k u_j^{i,k-1}$ in the formulation of Lemma \eqref{lem:MINLP_formulation}  by a new variable $\mathop{s_{lj}^{i,k}\in [-1,1]}$. To ensure that 
\[
w_{lj}^k u_j^{i,k-1} = s_{lj}^{i,k}
\]
holds, we have to add the set of inequalities
\begin{align*}
& s_{lj}^{i,k} \le u_j^{i,k} \\
& s_{lj}^{i,k} \ge -u_j^{i,k}\\
& s_{lj}^{i,k} \le w_{lj}^k + (1-u_j^{i,k})\\
& s_{lj}^{i,k} \ge w_{lj}^k - (1-u_j^{i,k}) .
\end{align*}
Note that if $u_j^{i,k}=0$, then the first two constraints ensure that $s_{lj}^{i,k}=0$. Since $w_{lj}^k\in [-1,1]$ this combination is also feasible for the last two constraints. If $u_j^{i,k}=1$, then the last two constraints ensure, that $s_{lj}^{i,k}=w_{lj}^k$, while $s_{lj}^{i,k}$ and $u_j^{i,k}$ are still feasible for the first two constraints.
\end{proof}

\begin{remark}
The regression variant can be modeled by replacing the $u^{i,K}$ variables by continuous variables $\mathop{z^i\in \R^{d_K}}$ for each $i\in [m]$, replacing the objective function \eqref{eq:MINLP_objective} by 
\[
\sum_{i=1}^{m} \ell_r\left( x^i , z^i\right)
\] 
and replacing Constraints \eqref{eq:MINLP_layers}-\eqref{eq:MINLP_layers2} with $k=K$ by the constraints 
\[
z^i = W^K u^{i,k-1} .
\]
\end{remark}

\begin{remark}
Besides others the following classical loss functions can be considered:
\begin{itemize}
	\item To model the empirical error we simply define the loss function by
	\[
	\sum_{i=1}^{m}\mathbbm{1}_{\left\{z^i\neq y^i\right\}} = \sum_{i=1: y^i=0}^{m}z^i + \sum_{i=1: y^i=1}^{m}(1-z^i)
	\]
	In this case we obtain a mixed-integer linear problem (MILP) in Theorem \ref{thm:MILP_formulation}.
	\item Theorem \ref{thm:MILP_formulation} shows that the regression variant with quadratic loss $\ell_r (x, z ) = \| x - z\|_2^2$ can be modelled as a mixed-integer quadratic problem (MIQP).
\end{itemize}
\end{remark}

Problem (MIP) can be solved by any off-the-shelf solvers like Gurobi for many classical loss functions \cite{gurobi}. Unfortunately, the number of integer variables of the MIP formulation is $\mathcal O(DKm)$, where $D$ is the maximum dimension of the layers, and therefore grows linear with the number of training samples and with the number of layers. For practical applications requiring large training sets solving the MIP formulation can be a hard or even impossible task. To tackle these difficulties we propose a fast heuristic in Section \ref{sec:heuristic}. Despite the latter drawback the MIP formulation has a lot of advantages and can give further insights into the analysis of deep neural networks:

\begin{itemize}
	\item The BDNN with binary weights can be modeled by simply setting the weight variables to be in $\{ 0,1\}$. The linearization results of Theorem \eqref{thm:MILP_formulation} apply also in this case.
	
	\item  Multiclass classification tasks can be modeled by using integer variables $u^{i,K}_j\in \mathbb Z\cap [a,b]$ in the last layer.
	
	\item More general discrete activation functions of the form
	\[
	\sigma^k(\alpha) = 
	v \ \text{ if } \underline{\lambda_k^{v}}\le \alpha \le \overline{\lambda_k^v} , \ v\in V\subset\mathbb Z
	\]
	for a finite set $V$ and pairwise non-intersecting intervals $[\underline{\lambda_k^{v}},\overline{\lambda_k^v} ]$ can be modeled by adding copies $u^{i,k,v}$ of the $u^{i,k}$ variables for each $v\in V$ and adding the constraints
	\begin{align*}
	& W^k \left( \sum_{v\in V} vu^{i,k-1,v}\right) \le M_k \left(1-u^{i,k,v}\right) + \lambda_k^{v} \\
	& W^k \left( \sum_{v\in V} vu^{i,k-1,v}\right) \ge -M_k \left(1-u^{i,k,v}\right) + \lambda_k^{v}
	\end{align*}
	for each $v\in V$, $k\in [K]\setminus\{1\}$ and $i\in[m]$. The two constraints for the first layer are defined similarly, replacing $\left( \sum_{v\in V} vu^{i,k-1,v}\right)$ by $x^i$. Note that the values $\underline{\lambda_k^{v}},\overline{\lambda_k^v}$ either have to be fixed in advance for each $v\in V$ or additional constraints have to be added which ensure the interval structure.

	\item The MIP formulation can easily be adjusted for applications where further constraints are desired. E.g. sparsity constraints of the form
	\[
	\| W^k\|_0\le s
	\]
	can be easily added to the formulation. Here $\| W^k\|_0$ is the number of non-zero entries of $W^k$. 
	
	\item Any classical approaches handling uncertainty in the data can be applied to the MIP formulation. In Section \ref{sec:adversarialAttacks} we will apply a robust optimization approach to the MIP formulation.
	
	\item The model is very flexible regarding changes in the training set. To add new data points that were not yet considered we just have to add the corresponding variables and constraints for the new data points to our already existing model and restart a solver, which is based on the idea of online machine learning.
	
	\item Classical solvers like Gurobi use branch \& bound methods to solve MIP formulations. During these methods at each time the optimality gap, i.e. the percental difference between the best known upper and lower bound, is known. These methods can be stopped after a certain optimality gap is reached.
\end{itemize}

\subsection{Heuristic Algorithm}\label{sec:heuristic}
In this section, we present a heuristic algorithm that is based on local search applied to the non-linear formulation in Lemma \ref{lem:MINLP_formulation}. This method is also known under the name \textit{Mountain-Climbing method}; see \cite{nahapetyan2009bilinear}. The idea is to avoid the quadratic terms by alternately optimizing the MINLP formulation over a subset of the variables and afterward over the complement of variables. Since for given weight variables exactly one feasible solution for the $u$-variables exists, this procedure would terminate after one iteration if all $u$-variables are contained in one of the problems. To avoid this problem we go through all layers and alternately fix the $u$ or the $W$-variables. The second problem uses exactly the opposite variables for the fixation. More precisely we iteratively solve the following two problems: 
\begin{equation}\label{eq:HeuristicWProb}\tag{$H1$}
\begin{aligned}
\min \ & \sum_{i=1}^{m} \ell\left( y^i , u^{i,K}\right)  \\
s.t. \quad & W^1 x^i < M_1 u^{i,1} + \lambda_1 \ \ \forall i\in [m] \\
& W^1 x^i \ge M_1 (u^{i,1}-1) + \lambda_1 \ \ \forall i\in [m]\\
&W^k u^{i,k-1} < M_k u^{i,k} + \lambda_k \\ 
&\qquad\qquad\forall i\in [m], k\in [K]\setminus \{1\}\\
&W^k u^{i,k-1} \ge M_k(u^{i,k}-1) + \lambda_k \\ 
&\qquad\qquad\forall i\in [m], k\in [K]\setminus \{1\}\\
& W^k\in[-1,1]^{d_k\times d_{k-1}}, \lambda_k\in [-1,1] \\ 
&\qquad\qquad\forall k\in [K]: k \text{ odd}\\
&u^{i,k}\in \left\{ 0,1\right\}^{d_k} \\ 
&\qquad\qquad\forall i\in [m], \ k\in [K-1]: k \text{ odd} \\
& u^{i,K}\in \left\{ 0,1\right\}^{d_K} \ \ \forall i\in [m].
\end{aligned}
\end{equation}
and
\begin{equation}\label{eq:HeuristicUProb}\tag{$H2$}
\begin{aligned}
\min \ & \sum_{i=1}^{m} \ell\left( y^i , u^{i,K}\right)  \\
s.t. \quad & W^1 x^i < M_1 u^{i,1} + \lambda_1 \ \ \forall i\in [m] \\
& W^1 x^i \ge M_1 (u^{i,1}-1) + \lambda_1 \ \ \forall i\in [m]\\
&W^k u^{i,k-1} < M_ku^{i,k} + \lambda_k \\ 
&\qquad\qquad\forall i\in [m], k\in [K]\setminus \{1\}\\
&W^k u^{i,k-1} \ge M_k(u^{i,k}-1) + \lambda_k \\ 
&\qquad\qquad\forall i\in [m], k\in [K]\setminus \{1\}\\
& W^k\in[-1,1]^{d_k\times d_{k-1}}, \lambda_k\in [-1,1] \\ 
&\qquad\qquad\forall k\in [K]: k \text{ even}\\
&u^{i,k}\in \left\{ 0,1\right\}^{d_k} \\ 
&\qquad\qquad\forall i\in [m], \ k\in [K-1]: k \text{ even} \\
& u^{i,K}\in \left\{ 0,1\right\}^{d_K} \ \ \forall i\in [m].
\end{aligned}
\end{equation}
In Problem \eqref{eq:HeuristicWProb} all variables in 
\begin{equation*}
	V_{\text{fix}}^1:=\{W^k,\lambda_k, u^{i,k} \text{ where $k\neq K$ and $k$ is even}\}
\end{equation*}
are fixed to given values; while in Problem \eqref{eq:HeuristicUProb} all variables in 
\begin{equation*}
	V_{\text{fix}}^2:=\{W^k,\lambda_k, u^{i,k} \text{ where $k\neq K$ and $k$ is odd}\}
\end{equation*}
are fixed to given values. In the heuristic procedure the fixed variables are always set to the optimal value of the preceding problem. Note that both problems are linear mixed-integer problems with roughly half of the variables of Problem (MIP). The heuristic is shown in Algorithm \ref{alg:heuristic}. Note that Algorithm \ref{alg:heuristic} terminates after a finite number of steps, since there only exist finitely many possible variable assignments for the variables $u^{i,K}$, and therefore only finitely many objective values exist.

\begin{algorithm}\caption{(Local Search Heuristic)}\label{alg:heuristic}
\begin{algorithmic}
\STATE {\bfseries Input:} $X\times Y$, $K$, $d_0,\ldots ,d_K$
	\STATE {\bfseries Output:} Weights $W^1, \ldots, W^K$.
	\STATE Draw random values for the variables in $V_{\text{fix}}^1$
	\REPEAT 
	\STATE Calculate optimal solution of \eqref{eq:HeuristicWProb}, if feasible, for the current fixations in $V_{\text{fix}}^1$.
	\STATE Set the values of all variables in $V_{\text{fix}}^2$ to the corresponding optimal values of \eqref{eq:HeuristicWProb}.
	\STATE Calculate optimal solution of \eqref{eq:HeuristicUProb} if feasible for the current fixations in $V_{\text{fix}}^2$.
	\STATE Set the values of all variables in $V_{\text{fix}}^1$ to the corresponding optimal values of \eqref{eq:HeuristicUProb}.
	\UNTIL{no better solution is found} 
	\STATE Return: $W=\{ W^{k}\}_{k\in [K]}$
\end{algorithmic}
\end{algorithm}

\section{Data Uncertainty}\label{sec:adversarialAttacks}
In this section we consider labeled data \[X\times Y=\left\{ (x^i,y^i) \ | \ i\in [m]\right\}\] where the euclidean norm of each data point is bounded by $r>0$ and the data points are subject to uncertainty, i.e. a true data point $x^i\in\R^n$ can be disturbed by an unknown deviation vector $\delta\in\R^n$. One approach to tackle data uncertainty is based on the idea of robust optimization and was already studied for regression problems and support vector machines in \cite{bertsimas2019robust,xu2009robustness,xu2009robust}. In the robust optimization setting, we assume that for each data point $x^i$ we have a convex set of possible deviation vectors $U^i\subset\R^n$ called \textit{uncertainty set} which is defined by
\[
U^i = \left\{ \delta\in \R^n \ | \ \|\delta\|\le r_i\right\}
\]
for a given norm $\|\cdot \|$ and radii $r_1,\ldots ,r_m$. Note that classical convex sets like boxes, polyhedrons, or ellipsoids can be modeled as above by using the $\ell_1$ or $\ell_2$ norm. The task is to find the weights of a neural network which are robust against all possible data disturbances in the uncertainty set $U:=U^1\times \cdots\times U^m$, i.e. we want to find weights $W^1,\ldots,W^k$ which minimize the worst-case loss over the training set and which are feasible for all possible disturbances. The difference to the situations studied in \cite{bertsimas2019robust,xu2009robustness,xu2009robust} is that the variables $u^{i,k}$ in our model, which are representing the value of the activation function, should be determined after the disturbance is known since the disturbance influences the data point and therefore the value of the activation function. Hence we can interpret the decision as a function of the uncertain parameters, i.e. $u^{i,k}: U^i \to \{ 0,1\}$. A useful approach to tackle this situation is called two-stage robustness (or adaptive robustness) and was already studied intensively in the literature \cite{buchheim2018robust}. Applied to our model the two-stage robust counterpart is

\begin{multline}\label{eq:TwoStageRobustVersion}
\min_{W^k,\lambda_k} ~\max_{\delta\in U} ~\min_{u^{i,k}}\bigg{\{} \sum_{i=1}^{m} \ell\left( y^i , u^{i,K}\right): W^1,\ldots ,W^k, \\ \lambda_1,\ldots ,\lambda_k, u^{1,1},\ldots ,u^{m,K} \in P(X,\delta)\bigg{\}}.
\end{multline}
where $P(X,\delta)$ is the set of feasible solutions of the inequality system
\begin{align*}
& W^1 (x^i+\delta^i) <  M_1^i u^{i,1} + \lambda_1 \ \ \forall i\in [m] \\
& W^1 (x^i+\delta^i) \ge M_1^i (u^{i,1}-1) + \lambda_1 \ \ \forall i\in [m]\\
&W^k u^{i,k-1} < M_k u^{i,k} + \lambda_k \ \ \forall i\in [m], k\in [K]\setminus \{1\}\\
&W^k u^{i,k-1} \ge M_k(u^{i,k}-1) + \lambda_k \ \ \forall i\in [m], k\in [K]\setminus \{1\}\\
& W^k\in[-1,1]^{d_k\times d_{k-1}}, \lambda_k\in [-1,1] \ \ \forall k\in [K]\\
&u^{i,k}\in \left\{ 0,1\right\}^{d_k} \ \ \forall i\in [m], \ k\in [K] 
\end{align*}
and $M_1^i:= n(r+r_i)$. The only difference to Constraints \eqref{eq:MINLP_firstlayer} -- \eqref{eq:MINLP_variables} is the appearance of the disturbance vectors $\delta^i$ in the first layer.

The idea is that in the first-stage the variables $W^k$ and $\lambda^k$ have to be calculated before the disturbance $\delta$ is known. For each possible first-stage solution the worst-case objective value overall scenarios $(\delta^1,\ldots, \delta^m)\in U$ is considered in the objective function. In the second-stage, after the scenario is known, the best second-stage decision for the variables $u^{i,k}$ is made. Since for each combination of first-stage variables and scenarios in $U$ exactly one second-stage variable is feasible, the second-stage minimization problem is equivalent to just finding this unique feasible solution, i.e. checking for each neuron if it is activated under the given disturbance or not. Problems of this form can be solved by so-called \textit{column-and-constraint algorithms} \cite{zeng2013solving}. The idea of this method is to start with a finite subset of scenarios in $U$ and iteratively add new scenarios to the problem. The new scenario is determined by an adversarial problem. This scenario is then added to the master-problem together with a copy of all second-stage variables, which will define the second-stage reaction to the new scenario. In this method, an increasing lower bound of Problem \eqref{eq:TwoStageRobustVersion} is given by the master-problem, while the adversarial problem provides an upper bound to the problem. Therefore in each iteration, the method provides an upper bound on the optimality gap. Unfortunately, this method performs very bad for discrete optimization problems even if the uncertainty only affects the objective function \cite{kammerling2020oracle}. Since to date no appropriate alternative methods exist, algorithmic progress in the field of two-stage robustness is desired to solve Problem \eqref{eq:TwoStageRobustVersion} on realistic data sets.

\section{Computations}\label{sec:computations}
In this section, we computationally compare the BDNN to the classical DNN with the ReLU activation function. We study networks with one hidden layer of dimension $d_1$. All solution methods were implemented in Python 3.8 on an Intel(R) Core(TM) i5-4460 CPU with 3.20GHz and 8 GB RAM. The classical DNN was implemented by using the Keras API where we used the ReLU activation function on the hidden layer and the Softmax on the output layer. We used the binary cross entropy loss function. The number of epochs was set to $100$. The exact IP formulation is given in Theorem \ref{thm:MILP_formulation} and all IP formulations used in the local search heuristic were implemented in Gurobi 9.0 with standard parameter settings. The strict inequalities in the IP formulations were replaced by non-strict inequalities adding $-0.0001$ to the right-hand-side. For the IP formulations, we set a time limit (wall time) of 24 hours.

\begin{figure}[h!]
\centering
\includegraphics[scale=0.5]{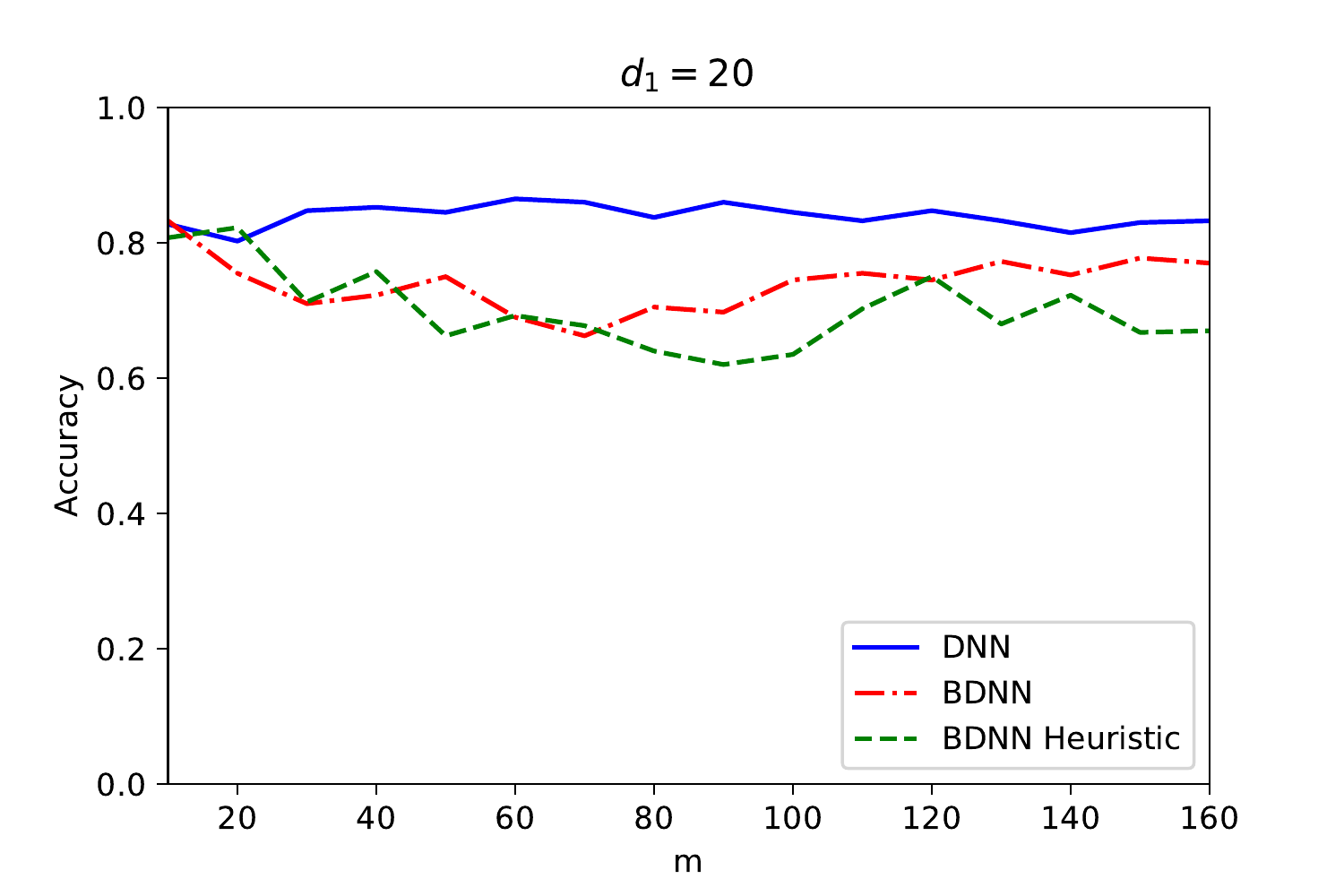}
\includegraphics[scale=0.5]{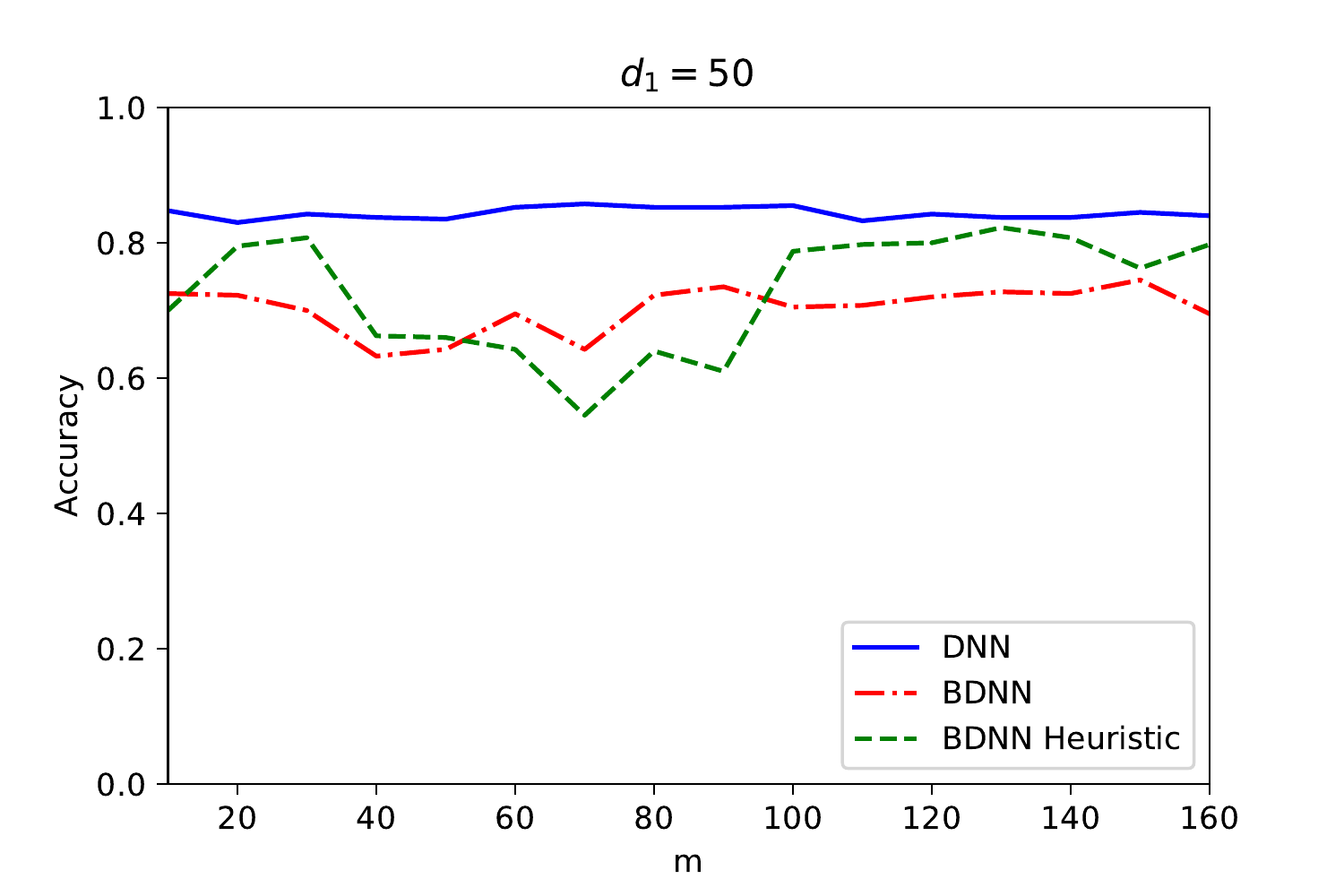}
\includegraphics[scale=0.5]{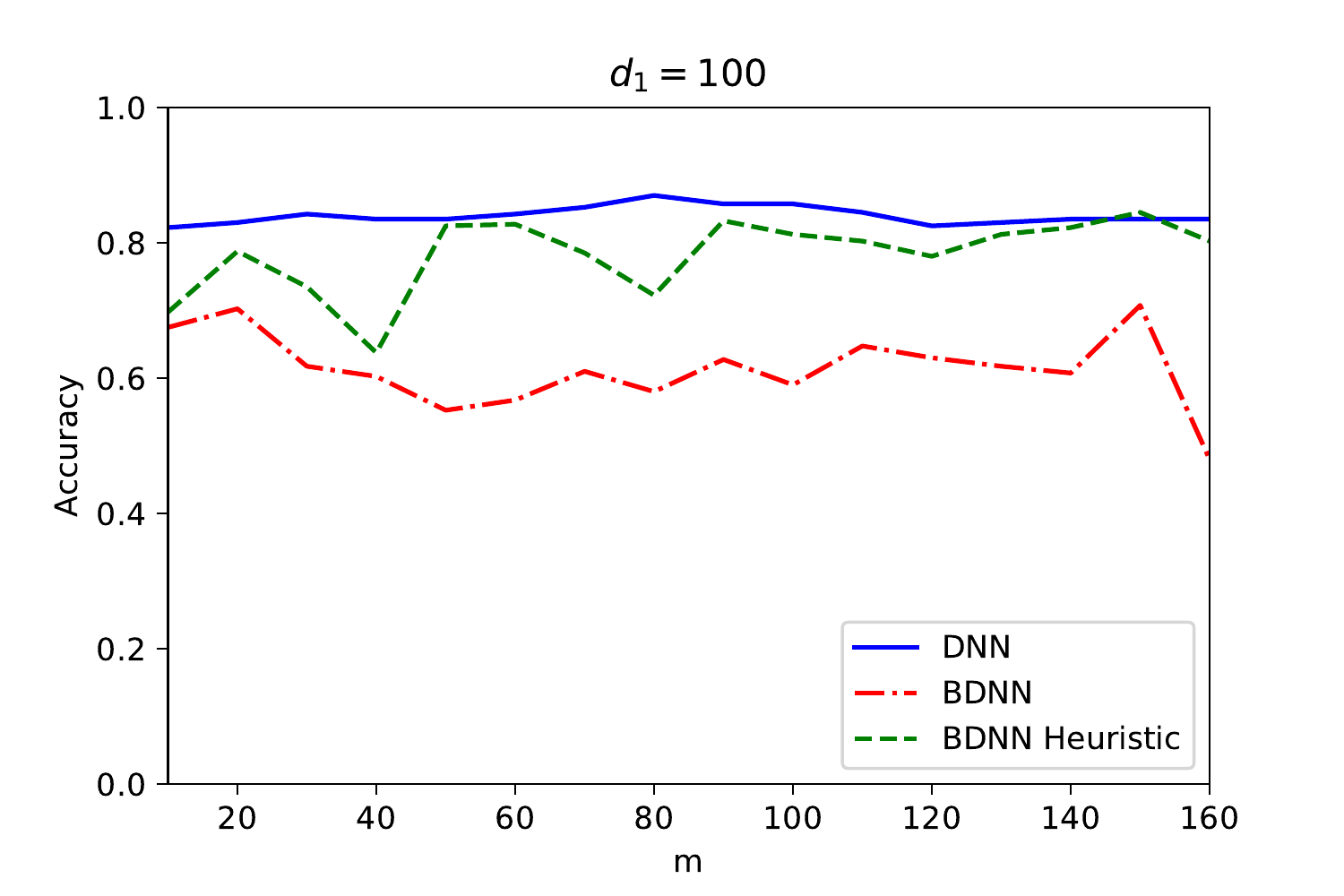}
\caption{Average accuracy over $10$ random instances for networks with one hidden layer of dimension $d_1$ trained on $m$ data points.}
\label{fig:plotsAcc}
\end{figure}

We generated $10$ random datasets in dimension $N=100$ each with $m=200$ data points. The entries of the data points were drawn from a uniform distribution with values in $[0,10]$ for one-third of the data points, having label $1$, and with values in $[-10,0]$ for the second third of the data points, having label $0$. The remaining data points were randomly drawn with entries in $[-1,1]$ and have randomly assigned labels. We split each dataset into a training set of $160$ samples and a testing set of $40$ samples. All computations were implemented for neural networks with one hidden layer of dimension $d_1\in \{ 20,50,100\}$. Figure \ref{fig:plotsAcc} shows the average classification accuracy on the testing set over all $10$ datasets achieved by the methods trained on $m$ of the training points. The results indicate that the exact and the heuristic version of the BDNN have lower accuracy than the DNN. Furthermore, the performance of both BDNN methods seem to be much more unstable and depend more on the choice of the training set. Interestingly for a hidden layer of dimension $100$, the heuristic method performs much better and can even compete with the classical DNN. In Figure \ref{fig:plotsTime} we show the runtime of all methods over $m$. Clearly, the runtime of the BDNN methods is much higher and seems to increase linearly in the number of data points. For real-world data sets with millions of data points, this method will fail using state-of-the-art solvers. Surprisingly, the runtime of the heuristic algorithm seems to be nearly the same as for the exact version, while the accuracy can be significantly better.

\begin{figure}[h!]
\centering
\includegraphics[scale=0.5]{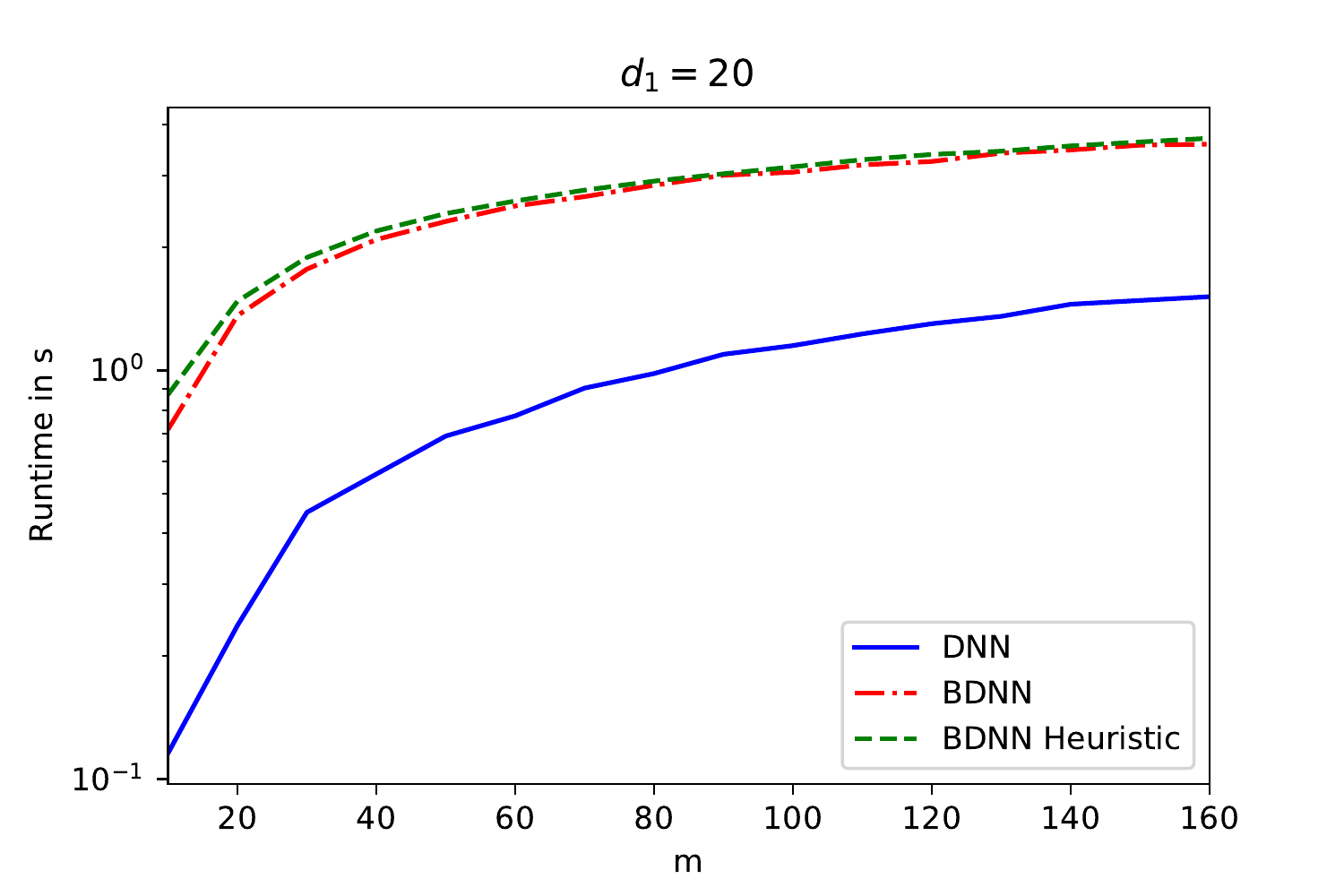}
\includegraphics[scale=0.5]{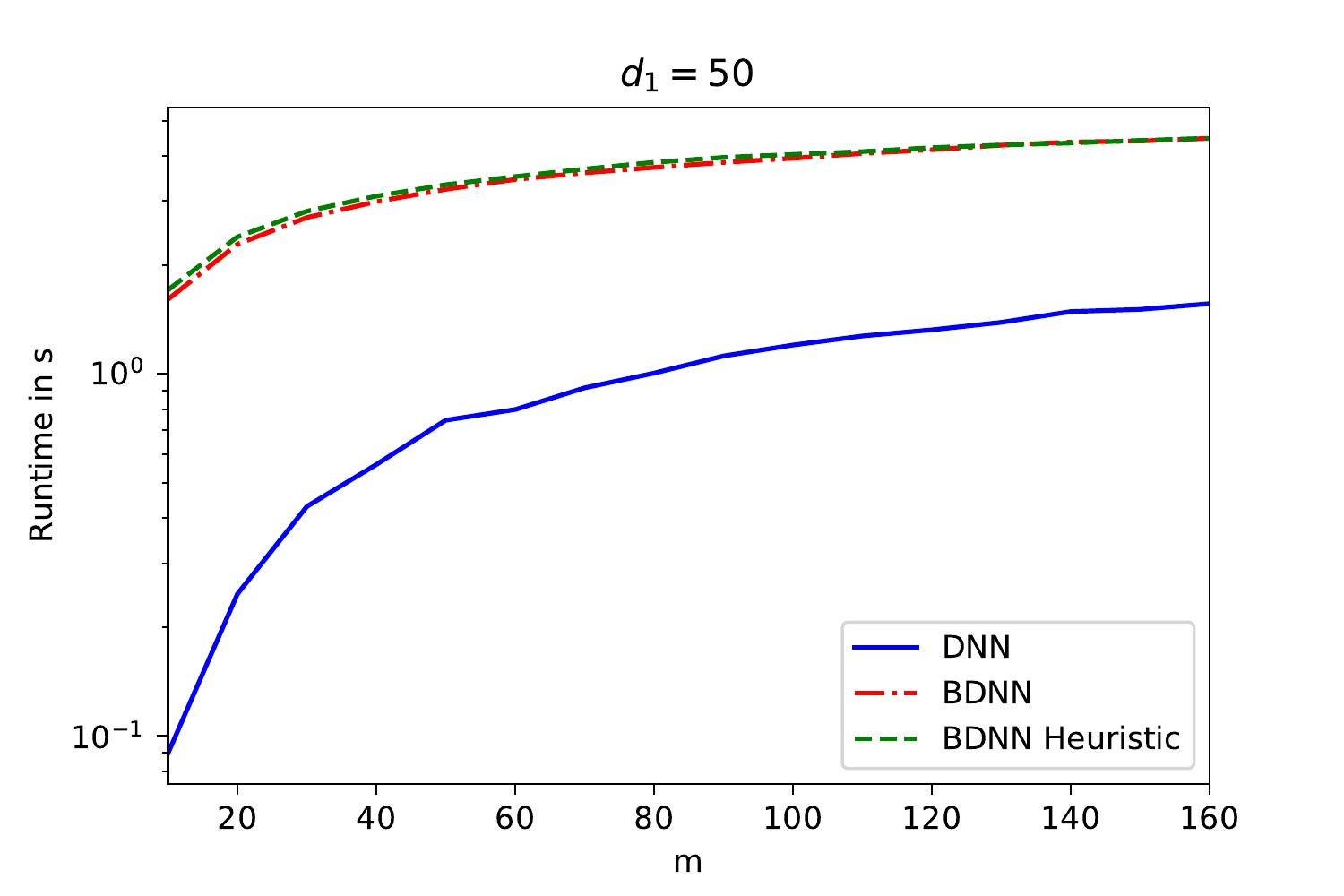}
\includegraphics[scale=0.5]{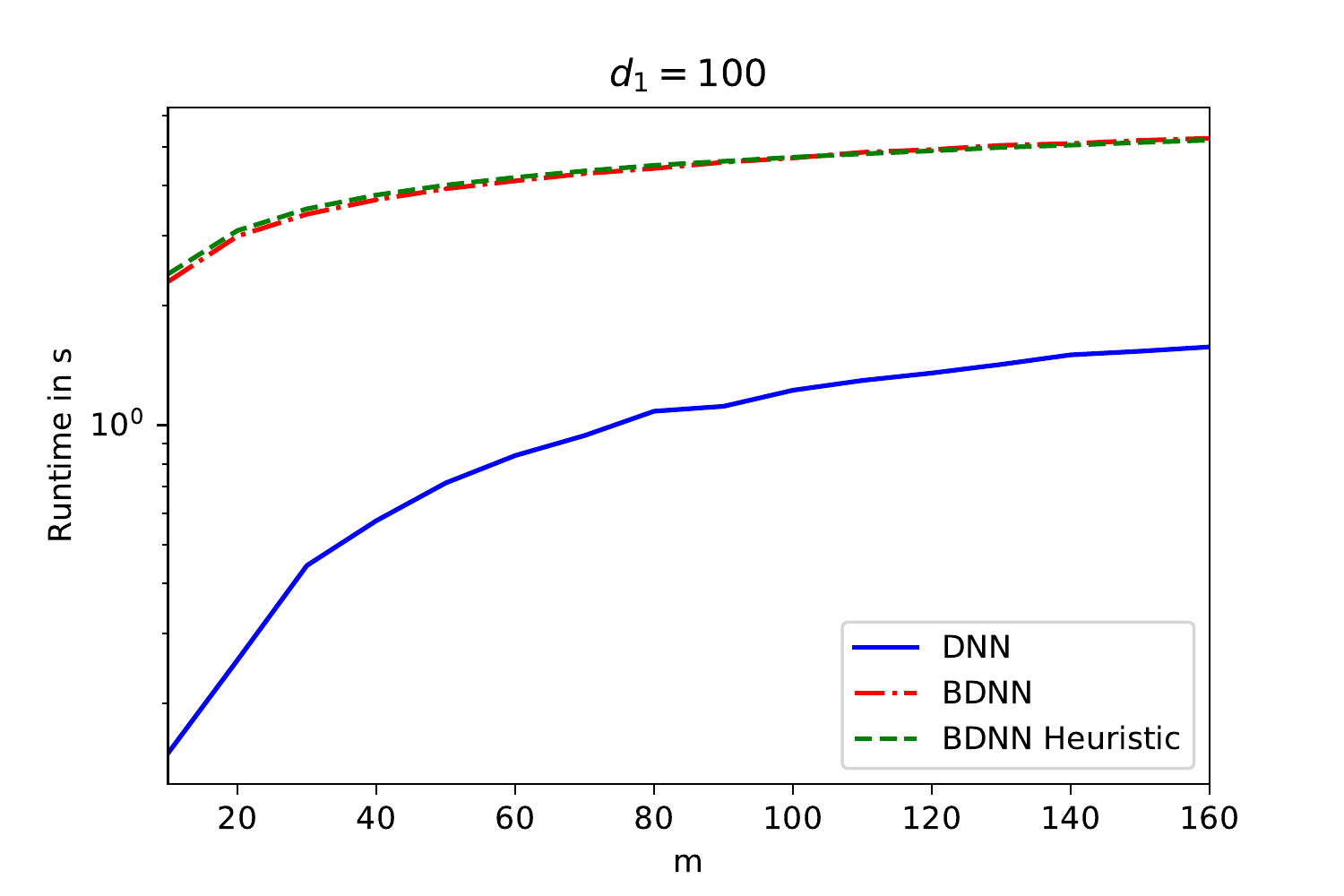}
\caption{Average runtime (logarithmic scale) over $10$ random instances for networks with one hidden layer of dimension $d_1$ trained on $m$ data points.} 
\label{fig:plotsTime}
\end{figure}

Additionally, we study all methods on the Breast Cancer Wisconsin dataset (BCW) \cite{Dua:2019}. Here we also test the BDNN version where the values $\lambda_k$ are all set to $0$ instead of being part of the variables; we indicate this version by BDNN$_0$. The dataset has $m=699$ data points with $N=9$ attributes which we split into $80\%$ training data and $20\%$ testing data.  Again all computations were implemented for neural networks with one hidden layer of dimension $d_1\in \{ 25,50\}$. In Tables \ref{tbl:breastcancerAcc} and \ref{tbl:breastcancerRest} we show the accuracy, precision, recall and F1 score of all methods on the BCW dataset for a fixed shuffle of the data returned by the scikit-learn method \textit{train\_test\_split} with the seed set to $42$. It turns out that the exact BDNN performs better if the values $\lambda_k$ are set to $0$, while the heuristic version performs better if the $\lambda_k$ are trained. The heuristic version of the BDNN has the best performance for $d_1=25$, significantly better than the DNN. For $d_1=50$ the DNN is slightly better. Nevertheless, the best accuracy of $95\%$ for the BCW dataset was achieved by the heuristic BDNN for $d_1=25$. In Table \ref{tbl:breastcancerShuffles} we compare the heuristic BDNN to the DNN on $10$ random shuffles of the BCW dataset and record the average, maximum, and minimum accuracies over all $10$ shuffles. It turns out that the heuristic BDNN outperforms the DNN with the best accuracy of $97.1\%$.

No results are reported for the robust BDNN version, presented in Problem \eqref{eq:TwoStageRobustVersion} since the column-and-constraint algorithm was not able to terminate during 24 hours on the random instances even for a network with $d_1=20$ and $20$ data points. Further future research regarding more efficient optimization methods for discrete robust two-stage problems is necessary to solve Problem \eqref{eq:TwoStageRobustVersion}. 

\begin{table}[h!]
\caption{Performance on the \textit{Breast Cancer Wisconsin} dataset.}
\label{tbl:breastcancerAcc}
\vskip 0.15in
\begin{center}
\begin{small}
\begin{sc}
\begin{tabular}{lc|cc}
\toprule
Method & $d_1$ & Acc. (\%) & Opt. Gap (\%)  \\
\midrule
BDNN    & 25 & 69.3 & 0.0 \\
BDNN$_0$ & 25 & 83.6 & 0.0\\
BDNN heur. & 25 & \textbf{95.0} & 0.0 \\
BDNN$_0$ heur. & 25 & 30.0 & 0.0 \\
DNN    & 25 & 91.4 & - \\
\hline
BDNN    & 50 & 69.3 & 0.0 \\
BDNN$_0$ & 50 & 84.3 & 0.51\\
BDNN heur. & 50 & 89.3 & 0.0 \\
BDNN$_0$ heur. & 50 & 71.4 & 0.0 \\
DNN    & 50 & \textbf{91.4} & - \\
\bottomrule
\end{tabular}
\end{sc}
\end{small}
\end{center}
\vskip -0.0in
\end{table}

\begin{table}[h!]
\caption{Performance on the \textit{Breast Cancer Wisconsin} dataset.}
\label{tbl:breastcancerRest}
\vskip 0.15in
\begin{center}
\begin{small}
\begin{sc}
\begin{tabular}{lc|ccc}
\toprule
Method & $d_1$ &  Prec. (\%) & Rec. (\%) & F1 (\%)  \\
\midrule
BDNN    & 25 & 48.0 & 69.3 & 56.7\\
BDNN$_0$ & 25 & 83.2 & 83.6 & 83.1\\
BDNN heur. & 25 & \textbf{95.0} & \textbf{95.0} & \textbf{95.0}\\
BDNN$_0$ heur. & 25 & 38.7 & 30.0 & 17.6\\
DNN    & 25 & 91.8 & 91.4 & 91.5\\
\hline
BDNN    & 50 & 63.6 & 69.3 & 58.0\\
BDNN$_0$ & 50 & 86.4 & 84.3 & 84.7\\
BDNN heur. & 50 & \textbf{91.5} & 89.3 & 89.6\\
BDNN$_0$ heur. & 50 & 74.6 & 71.4 & 72.3 \\
DNN    & 50 & 91.4 & \textbf{91.4} & \textbf{91.3}\\
\bottomrule
\end{tabular}
\end{sc}
\end{small}
\end{center}
\vskip -0.0in
\end{table}

\begin{table}[h!]
\caption{Accuracy on the \textit{Breast Cancer Wisconsin} dataset over $10$ random shuffles of the data.}
\label{tbl:breastcancerShuffles}
\vskip 0.15in
\begin{center}
\begin{small}
\begin{sc}
\begin{tabular}{lc|ccc}
\toprule
Method & $d_1$ &  Avg. (\%) & Max (\%) & Min (\%)  \\
\midrule
BDNN heur. & 25 & \textbf{93.2} & \textbf{97.1} & 85.0\\
DNN    & 25 & 89.1 & 91.4 & \textbf{85.7}\\
\bottomrule
\end{tabular}
\end{sc}
\end{small}
\end{center}
\vskip -0.0in
\end{table}

\section{Conclusion}
We show that binary deep neural networks can be modeled by mixed-integer programming formulations, which can be solved to global optimality by classical integer programming solvers. Additionally, we present a heuristic algorithm and a robust version of the model. Our tests on random and real datasets indicate that binary deep neural networks are much more unstable and computationally harder to solve than classical DNNs. Nevertheless, the heuristic algorithm outperforms the classical DNN on the Breast Cancer Wisconsin dataset. Moreover, the mixed-integer programming formulation is very adjustable to variations of the model and could give new insights into the understanding of deep neural networks. This motivates further research into the scalability of the IP method and the study of other tractable reformulations of DL algorithms.

\small
\bibliography{icmlref}
\bibliographystyle{icml2020}

\end{document}